\documentclass{article}  

\usepackage{setspace}
\usepackage{float}
\usepackage{amsmath,amsfonts,amssymb}
\usepackage{amsthm}
\usepackage{graphicx}
\usepackage{epstopdf}
\usepackage{caption}

\usepackage[margin=2.5cm]{geometry}
\usepackage{color}

\theoremstyle{plain}
\newtheorem{theorem}{Theorem}
\newtheorem{lemma}{Lemma}

\newtheorem{corollary}{Corollary}
\theoremstyle{definition}
\newtheorem{definition}{Definition}
\newtheorem{example}{Example}
\newtheorem{remark}{Remark}

\newcommand\blfootnote[1]{%
	\begingroup
	\renewcommand\thefootnote{}\footnote{#1}%
	\addtocounter{footnote}{-1}%
	\endgroup
}

\begin{document}	

\title{Finite time stability of tempered fractional systems with time delays\blfootnote{This 
is a preprint version of the paper published open access in 
\emph{Chaos Solitons Fractals}.}}

\author{Hanaa Zitane$^{1}$\\
\texttt{h.zitane@ua.pt} 
\and Delfim F. M. Torres$^{1,2,}$\thanks{Corresponding author: \texttt{delfim@ua.pt}}\\
\texttt{delfim@ua.pt}; \texttt{delfim@unicv.cv}}

\date{$^{1}$Center for Research and Development in Mathematics and Applications (CIDMA),
Department of Mathematics, University of Aveiro, 3810-193 Aveiro, Portugal\\[0.3cm]
$^{2}$Research Center in Exact Sciences (CICE), 
Faculty of Sciences and Technology (FCT), 
University of Cape Verde (Uni-CV), 
7943-010 Praia, Cape Verde}

\maketitle	

	
\begin{abstract}
We investigate the notion of finite time stability 
for tempered fractional systems (TFSs) with time delays 
and variable coefficients. Then, we examine some sufficient 
conditions that allow concluding the TFSs
stability in a finite time interval, which include the nonhomogeneous 
and the homogeneous delayed cases.  We present two different approaches. 
The first one is based on H\"older's and Jensen's inequalities, 
while the second one concerns the Bellman--Gr\"onwall method 
using the tempered Gr\"onwall inequality.  
Finally, we provide two numerical examples to show the 
practicability of the developed procedures.

\medskip


\medskip

\end{abstract}


\section{Introduction}

In the last decade, fractional calculus has gained increasing interests 
due to its crucial and efficient role in modeling various real world phenomena 
in different fields of science and engineering \cite{mittag,MR2984893,MR1219954}.
Fractional calculus involves the operation of convolution with a power law function. 
If one multiplies the fractional derivative and integral by an exponential term, 
the result will be a tempered fractional derivative (TFD) and integral \cite{Meerschaert,MR3342453}, 
which provide undeviating generalization to the existing Caputo and Riemann--Liouville fractional operators 
and have many merits, both mathematically and practically. A truncated L\'{e}vy flight was investigated 
to capture the natural cutoff in real physical systems \cite{MR1303317}. Without a sharp cutoff, 
the tempered L\'{e}vy flight was studied as a smoother alternative \cite{AE:5}. 
Cartea and del-Castillo-Negrete \cite{AE:6} explored the tempered fractional diffusion equation 
by the tempered L\'{e}vy flight. Furthermore, stochastic applications, such as tempered L\'{e}vy flights, 
present a complete set of statistical physics and numerical analysis tools including solving 
multi-dimensional partial differential equations \cite{Li2}. TFD can be also found in geophysics 
\cite{Meerschaert2}, Brownian motion \cite{Meerschaert3}, and so on.
For further information about tempered fractional calculus and its applications, 
we refer the reader to \cite{AE:6,Fernandez,Meerschaert} and references therein.

As in classical calculus, stability is still one of the most extensively 
studied subjects in control theory and fractional systems analysis 
\cite{Agarwal}. The stability analysis 
of fractional differential equations with the TFD
is at its initial stage. Recently, some stability results of tempered
fractional systems (TFSs) have been established in \cite{Deng, Gassara, Gu}. 
In \cite{Deng}, sufficient conditions that ensure the Mittag-Leffler stability of TFSs 
are investigated by means of a tempered fractional comparison principle and the extended 
Lyapunov direct method. Further, in \cite{Gassara}, the Lyapunov approach is applied 
to analyze the generalized practical Mittag-Leffler stability of a class of fractional 
nonlinear systems evolving TFDs. Moreover, in \cite{Gu}, the asymptotic and Mittag-Leffler 
stability of tempered fractional neural networks, with and without delay, 
are studied by using the Banach fixed point theorem.

Here, we consider the stability from a non-Lyapunov 
point of view, precisely, finite time stability. 
This approach concerns the system stability and, simultaneously, the bounds 
of its trajectories. In fact, a dynamical system could be stable but 
still entirely useless because of undesirable transient performances. 
Then, it may be important to consider the stability of such systems 
with regard to certain subsets of the state-space that are defined, a priori, 
in a given problem. From the engineering point of view, the boundedness properties 
of the system responses are very interesting. For example, constraining the state 
of a system in a transient regime to not exceed certain limits, to avoid 
saturations and excessive excitations of vital parts and nonlinear dynamics. 
In addition, it is of particular significance to analyze the behavior of dynamical 
systems over a finite time interval, especially when the systems 
lifetime is finite \cite{MR2288960,MR3229176,Moulay}.

To the best of our knowledge, the finite time stability problem of TFSs, 
including time delay ones, has not yet been analyzed. This motivated us to write the current paper. 
Therefore, in this work we propose to study the finite time stability for nonlinear TFSs with variable 
coefficients and time delays. Compared to \cite{Gu}, here we study the finite time 
stability concept of a more general class of nonlinear tempered delayed systems 
and also by means of two different methods that were not used yet for the stability 
of any class of tempered systems \cite{Deng,Gassara,Gu}. Our first method concerns 
the Bellman--Gr\"onwall approach using the tempered Gr\"onwall inequality. 
The second one is based on H\"older's and Jensen's inequalities. 
It is important to mention that the problem of finite time stability 
of fractional-order time systems, including the delayed cases, is widely investigated 
in the literature using different approaches \cite{Arthia,Lazarevi,LiWang,MaWu}. 
For example, based on Mittag-Leffler functions and the generalized Gr\"onwall
inequality, sufficient conditions 
that ensure the finite time stability of Caputo fractional order nonlinear systems 
with damping behavior are derived in \cite{Arthia}. Moreover, a finite time stability 
test procedure is presented for linear nonhomogeneous fractional time-delay systems 
though the Bellman-Gr\"onwall approach \cite{Lazarevi}. Also, the stability in the finite range 
of time for Caputo linear fractional delayed systems is studied by means of a delayed 
Mittag-Leffler type matrix \cite{LiWang}, while in \cite{MaWu}, criteria of 
finite time stability of Hadamard fractional differential linear and nonlinear equations 
in weighted Banach spaces are established using the method of successive approximations 
and Beesack's inequality with a weakly singular kernel.

The paper is organized as follows. In Section~\ref{sec:2}, we recall
necessary notions and results from the literature that will be useful 
in the sequel. Our original results are then given in Section~\ref{sec:3}:
we first establish the existence result for a class of nonlinear delayed TFSs, 
then we prove a delay-dependent- (Theorem~\ref{theoa})
and a delay-independent-criterion (Theorem~\ref{theob})
for the finite time stability of time delay nonhomogeneous TFSs
with variable coefficients. The usefulness of the proved criteria 
is illustrated in Section~\ref{sec:4} with two examples. We end with
Section~\ref{sec:5} of conclusions, pointing out some
possible future directions of research.


\section{Preliminaries}
\label{sec:2}

Tempered fractional operators, as we know them today, 
appear to have been introduced in \cite{AE:6}. 
However, other notions of ``tempered'' derivatives 
can be referred back to the seventies 
of the 20th century \cite{MR0365130}.
In this section, we first state some definitions 
and fundamental lemmas related to the tempered fractional order 
operators that are employed throughout this paper. 
For a more general tempered fractional calculus, 
that includes and generalizes what are usually called
substantial, tempered, and shifted fractional operators,
we refer to \cite{MR4279844}, where the reader
can also find a discussion if the tempered fractional derivative 
can be considered as a class of fractional derivatives or not
(see Section~3.4 of \cite{MR4279844}).

\begin{definition}[See \cite{Li,Meerschaert}]
Let $\alpha>0$, $\rho>0$ and $v$ be an absolutely integrable 
function defined on $[a,b]$, $a, b \in \mathbb{R}$, and $a<b$
(if $b=\infty$, then the interval is half-open). 
The tempered fractional integral of the function $v$ is defined as follows:
$$
{}^T\!I_{a}^{\alpha,\rho}v(t)=\dfrac{1}{\Gamma(\alpha)}
\int_{a}^{t}e^{-\rho(t-s)}(t-s)^{\alpha-1}v(s)\, \mathrm{d}s,
$$
where $\Gamma(\cdot)$ is the Euler Gamma function \cite{mittag} defined by
$$
\Gamma(r)= \displaystyle \int_{0}^{\infty}e^{-s}s^{r-1} 
\, \mathrm{d}s,\quad r\in\mathbb{C}.
$$	
\end{definition}

\begin{definition}[See \cite{Li,Meerschaert}] 
\label{GPF}
Let $\alpha\in(0,1)$ and $\rho>0$. The Caputo tempered fractional 
order derivative  of a function $v \in C^{1}([a,b], \mathbb{R})$ 
is given by
$$
{}^T\! D_{a}^{\alpha,\rho}v(t)=\dfrac{1}{\Gamma(1-\alpha)}
\int_{a}^{t}e^{-\rho(t-s)}(t-s)^{-\alpha}D^{1,\rho}v(s)\, \mathrm{d}s
$$
with 
$D^{1,\rho}v(t)=\rho v(t)+ dv'(t)$,
where $d=1$ and its dimension equals the dimension 
of the independent variable $t$.
\end{definition}

\begin{remark}
If $v(t)$ stands for some chemical, geometrical or physical 
quantity as a function of time $t$, 
for example, concentration, coordinate or momentum, then its combination 
with its derivative $v'(t)$ does not preserve the correct dimensions. 
In order to ensure the dimensional homogeneity within 
the combination $\rho v(t)+ dv'(t)$, we added the constant $d = 1$ 
with a dimension equal to the dimension of the independent variable $t$.
\end{remark}
 
\begin{remark}
In the case $\rho=0$, the TFD coincides with the 
left Caputo fractional derivative \cite{mittag}.	
\end{remark}

We have the following Gr\"onwall inequality 
in the framework of tempered fractional integral.

\begin{lemma}[Tempered Gr\"onwall inequality \cite{Gronwall}]
\label{Gronwall}
Suppose $\alpha>0$, $\rho>0$, $g$ and $f$ are nonnegative 
and locally integrable functions on $[0,t_{f})$ $(t_{f}\leq \infty)$ 
and $h$ is a nonnegative, nondecreasing,
and continuous function on $[0,t_{f})$ satisfying 
$h(t) \leq L$, where $L$ is a constant. Moreover, if
$$
g(t)\leq f(t)+ h(t)\displaystyle 
\int_{0}^{t}e^{-\rho(t-s)}(t-s)^{\alpha-1}g(s) \, \mathrm{d}s,
$$
then
$$
g(t)\leq f(t)+\displaystyle \int_{0}^{t} \left[
\sum_{n=1}^{+\infty}\dfrac{\left(h(t)\Gamma(\alpha)\right)^{n}}{\Gamma(n\alpha)}
e^{-\rho(t-s)}(t-s)^{n\alpha-1}f(s)\right]\, \mathrm{d}s,
\quad t\in[0,t_{f}].
$$
If, in addition, function $f$ is nondecreasing on $[0, t_{f})$, then
\begin{equation}
\label{GrenwalMittga}
g(t)\leq f(t)E_{\alpha}(h(t)\Gamma(\alpha)t^{\alpha}),
\quad t\in[0,t_{f}],
\end{equation}
where $E_{\alpha}(\cdot)$ is the Mittag-Leffler 
function of one parameter \cite{mittag} defined by 
$$
E_{\alpha}(s)=\displaystyle
\sum_{n=0}^{+\infty}\dfrac{s^{n}}{\Gamma(\alpha n+1)},
\quad s\in\mathbb{C}.
$$
\end{lemma}

We end this section by stating two inequalities that will 
be used to prove the finite time stability of nonlinear TFSs.

\begin{lemma}[H\"older's inequality \cite{Holder}]
\label{Holder}
Let $g,q>1$ such that $\dfrac{1}{g}+\dfrac{1}{q}=1$. 
If $|h(\cdot)|^{g},|k(\cdot)|^{q}\in L^{1}(M)$, 
then $h(\cdot)k(\cdot)\in L^{1}(M)$ and 
$$
\displaystyle \int_{M} |h(r)k(r)|\, \mathrm{d}r
\leq \left( \displaystyle \int_{M}   
|h(r)|^{g} \mathrm{d}r\right)^{\dfrac{1}{g}}
\left( \displaystyle \int_{M} 
|k(r)|^{q}\mathrm{d}r\right)^{\dfrac{1}{q}},
$$ 
where $L^{1}(M)$ designs the Banach space of all Lebesgue measurable 
functions $h:M\longrightarrow \mathbb{R}$ with 
$$
\displaystyle \int_{M}   |h(r)|^{p} \mathrm{d}r<\infty.
$$
\end{lemma}

\begin{lemma}[Jensen's inequality \cite{Kuczma}]
\label{Jensen}	
Let $n\in \mathbb{N}$ and $m_{1}, m_{2},\ldots,m_{n}$ 
be nonnegative real numbers. Then,
$$
\left(\sum_{k=1}^{n}m_{k}\right)^{p}\leq n^{p-1} \sum_{k=1}^{n}m^{p}_{k},
\quad~ \text{ for }~ p\geq 0.
$$
\end{lemma}


\section{Main Results}
\label{sec:3}

In this section, we shall analyze the finite time stability of the following 
class of time delay nonhomogeneous TFS with variable coefficients:
\begin{equation}
\label{system2}
\left\{
\begin{array}{ll}
{}^T\! D_{0}^{\alpha,\rho}y(t)
=e^{-\rho t}(Ay(t)+By(t-\tau)+f(t,y(t),y(t-\tau))), 
& t\in[0,T],\\
y(t)=\omega(t),  & t\in[-\tau,0],
\end{array}
\right.
\end{equation}
and the associated homogeneous TFS
\begin{equation}
\label{system1}
\left\{
\begin{array}{ll}
{}^T\! D_{0}^{\alpha,\rho}y(t)
=e^{-\rho t}(Ay(t)+By(t-\tau)), & t\in[0,T],\\
y(t)=\omega(t),  & t\in[-\tau,0],
\end{array}
\right.
\end{equation}
where $\alpha\in(0,1)$, $\rho \in(0,1]$, $A$ and $B$ are constant $n\times n$ matrices, 
$T>0$ is a real number, $\tau>0$ is a time delay, 
$\omega(\cdot)$ is a continuous function on $[-\tau,0]$ with the norm 
$\Arrowvert \omega \Arrowvert_{C}= \underset{t \in [-\tau,0]}{\sup} 
\Arrowvert \omega (t)\Arrowvert$ such that $\Arrowvert \cdot \Arrowvert$ is the maximum norm, 
and $f: [0,T]\times\mathbb{R}^{n}\times\mathbb{R}^{n}\longrightarrow \mathbb{R}^{n}$  
is a given nonlinear continuous function with $f(t,0,0)=0$.
 
We present the following definition of the finite 
time stability of system \eqref{system2}.

\begin{definition}
\label{def1:stab}
The system \eqref{system2} is finite time stable with regard 
to $\{\xi, \varepsilon, J\}$, $\xi\leq\varepsilon$, if 
\begin{equation}
\label{Cnd1}
\Arrowvert \omega \Arrowvert_{C} < \xi
\end{equation}
implies
\begin{equation}
\label{Cnd2}
\Arrowvert y(t) \Arrowvert < \varepsilon, \quad \forall t \in J,
\end{equation}
with $\xi$ and $\varepsilon$ positive real numbers; and 
$J$ is the time interval $J=[0,T]\subset\mathbb{R}$.
\end{definition}

\begin{remark}
When the system~\eqref{system2} is of order $\rho=0$, 
we retrieve the finite time stability definition 
of Caputo fractional delayed systems \cite{Lazarevi}.
\end{remark}

\begin{remark}
Intuitively, by defining two
bounded sets in the state space, called ``the initial set'' 
and ``the set of trajectories'',
the system is finite time stable if the trajectories of the system 
emanating from the initial set remain in the set of trajectories  
over a finite time interval.
\end{remark}

Existence and uniqueness of solution for the nonlinear tempered 
fractional delayed system~\eqref{system2} is stated by the following lemma.

\begin{lemma}
\label{Lemma:1}
Let $\alpha\in(0,1)$, $\rho \in(0,1]$, 
$f \in C([0, T]\times\mathbb{R}^{n}\times\mathbb{R}^{n}, \mathbb{R}^{n})$ and 
$\omega \in C([-\tau, 0], \mathbb{R}^{n})$.
The function $y : [-\tau, T]\longrightarrow\mathbb{R}^{n}$ 
is a solution of system~\eqref{system2} if, and only if, it satisfies
\begin{equation}
\label{system-2}
\left\{
\begin{array}{ll}
y(t)=\omega(0)e^{-\rho t}+{}^T\! I_{0}^{\alpha,\rho}
e^{-\rho s}\left[  Ay(t)+By(t-\tau)+f(t,y(t),y(t-\tau))\right], 
& t\in[0,T],\\
y(t)=\omega(t),  & t\in[-\tau,0].
\end{array}
\right.
\end{equation}
Moreover, if for any functions 
$y,z : [-\tau, T]\longrightarrow\mathbb{R}^{n}$, 
there exists a constant $L_{f}>0$ such that
\begin{equation}
\label{condLipch}
\Arrowvert f(t,y(t),y(t-\tau))- f(t,z(t),z(t-\tau))  \Arrowvert 
\leq L_{f}  \left(\Arrowvert  y(t)-z(t) \Arrowvert 
+ \Arrowvert  y(t-\tau)-z(t-\tau) \Arrowvert
\right),\quad t\in[0,T],
\end{equation}
then system~\eqref{system2} has a unique mild solution.
\end{lemma}

\begin{proof}
The proof follows by using similar techniques of Lemma~1 and Theorem~2 in \cite{SolGPFD} 
and by taking into account Definition~\ref{GPF} and the tempered 
Gr\"onwall inequality (Lemma~\ref{Gronwall}).
\end{proof}


\subsection{Time delay dependent criterion}

Here, we establish a delay-dependent criterion that enables us 
to check the finite time stability of the nonhomogeneous time delay TFSs
\eqref{system2}. Our proof uses H\"older's and Jensen's inequalities.

\begin{theorem}
\label{theoa}
Let $\xi, \varepsilon>0$ be given real numbers and consider 
$g=1+\alpha$ and $q=1+\dfrac{1}{\alpha}$.
The system \eqref{system2} is finite time stable with 
respect to $\{\xi, \varepsilon,J\}$, $\xi\leq\varepsilon$, 
if $f$ satisfies condition \eqref{condLipch} and
\begin{equation}
\label{cnd4}
^{q}\sqrt{\dfrac{3^{\frac{1}{\alpha}}q+(3^{\frac{1}{\alpha}}\varPsi+q\varPhi
+\varPsi\varPhi)e^{(\varPsi+q) t}}{q+\varPsi}}
\leq \dfrac{\varepsilon}{\xi}, \quad \forall t \in J,
\end{equation}
where
\begin{equation}
\label{Psi}
\varPsi=\dfrac{3^{\frac{1}{\alpha}}((\lambda_{\max}(A)+L_{f})^{q}
+(\lambda_{\max}(B)+L_{f})^{q}e^{-q\tau})V^{q}}{\Gamma^{q}(\alpha)},
\quad 
V=\left(\dfrac{\Gamma(\alpha^{2})}{g^{\alpha^{2}}}\right)^{1/g}
\end{equation}
and
\begin{equation}
\label{phi}
\varPhi=\dfrac{3^{\frac{1}{\alpha}}(\lambda_{\max}(B)
+L_{f})^{q}(1-e^{-\tau q})}{q\Gamma^{q}(\alpha)}V^{q}
\end{equation}
with $\lambda_{\max}(A)$ and $\lambda_{\max}(B)$ denoting 
the largest singular values of the matrices $A$ and $B$, respectively.
\end{theorem}

\begin{proof}
According to Lemma~\ref{Lemma:1}, 
the solution of system \eqref{system2} can be written as 
$$
y(t)=\omega(0)e^{-\rho t}
+\dfrac{1}{\Gamma(\alpha)}
\displaystyle \int_{0}^{t}e^{-\rho (t-s)}
e^{-\rho s}(t-s)^{\alpha-1}\left[  
A(s)y(s)+B(s)y(s-\tau)+f(s,y(s),y(s-\tau))\right] \, \mathrm{d}s.
$$ 
It follows that	
\begin{equation}
\Arrowvert y(t)\Arrowvert
\leq\Arrowvert\omega(0)\Arrowvert+\dfrac{1}{\Gamma(\alpha)}
\displaystyle\int_{0}^{t}(t-s)^{\alpha-1}\left[\lambda_{\max}(A)
\Arrowvert y(s)\Arrowvert+\lambda_{\max}(B)\Arrowvert y(s-\tau)\Arrowvert
+\Arrowvert f(s,y(s),y(s-\tau))\Arrowvert\right] \, \mathrm{d}s.
\end{equation}
Using condition \eqref{condLipch} with $f(s,0,0)=0$, it implies
\begin{equation*}
\Arrowvert y(t)\Arrowvert\leq\Arrowvert\omega(0)\Arrowvert
+\dfrac{1}{\Gamma(\alpha)}\displaystyle\int_{0}^{t}(t-s)^{\alpha-1}
\left[(\lambda_{\max}(A)+L_{f})\Arrowvert y(s)\Arrowvert
+(\lambda_{\max}(B)+L_{f})\Arrowvert y(s-\tau)\Arrowvert\right] \, \mathrm{d}s,
\end{equation*}
which yields
\begin{equation*}
\Arrowvert y(t)\Arrowvert
\leq\Arrowvert\omega(0)\Arrowvert+\dfrac{\lambda_{\max}(A)
+L_{f}}{\Gamma(\alpha)}\displaystyle\int_{0}^{t}(t-s)^{\alpha-1}e^{s}e^{-s}
\Arrowvert y(s)\Arrowvert\, \mathrm{d}s
+\dfrac{\lambda_{\max}(B)+L_{f}}{\Gamma(\alpha)}
\displaystyle\int_{0}^{t}(t-s)^{\alpha-1}e^{s}e^{-s}
\Arrowvert y(s-\tau)\Arrowvert \, \mathrm{d}s.
\end{equation*}
Applying H\"older's inequality (Lemma~\ref{Holder}) 
with $g=1+\alpha$ and $q=1+\dfrac{1}{\alpha}$, one obtains that
\begin{multline}
\label{imp2}	
\Arrowvert y(t)\Arrowvert
\leq\Arrowvert\omega(0)\Arrowvert+\dfrac{\lambda_{\max}(A)
+L_{f}}{\Gamma(\alpha)}\left(\displaystyle
\int_{0}^{t}(t-s)^{g(\alpha-1)}e^{gs}\, 
\mathrm{d}s\right)^{\dfrac{1}{g}} \times \left( 
\displaystyle\int_{0}^{t} e^{-qs}\Arrowvert y(s)\Arrowvert^{q}\, 
\mathrm{d}s\right)^{\dfrac{1}{q}}\\
+\dfrac{\lambda_{\max}(B)+L_{f}}{\Gamma(\alpha)}\left(
\displaystyle\int_{0}^{t}(t-s)^{g(\alpha-1)}e^{gs}\, 
\mathrm{d}s\right)^{\dfrac{1}{g}} \times \left( 
\displaystyle\int_{0}^{t} e^{-qs}\Arrowvert y(s-\tau)\Arrowvert^{q}\,
\mathrm{d}s\right)^{\dfrac{1}{q}}.	
\end{multline}
It is easy to show that
\begin{equation}
\label{imp3}
\displaystyle\int_{0}^{t}(t-s)^{g(\alpha-1)}e^{gs}\, \mathrm{d}s 
\leq \dfrac{e^{gt}\Gamma(g(\alpha-1)+1)}{g^{g(\alpha-1)+1}}
=\dfrac{\Gamma(\alpha^{2})e^{gt}}{g^{\alpha^{2}}}.
\end{equation}
Combining inequality \eqref{imp2} and relation \eqref{imp3}, one has
\begin{equation*}	
\begin{split}
\Arrowvert y(t)\Arrowvert
&\leq \Arrowvert\omega(0)\Arrowvert+\dfrac{(\lambda_{\max}(A)
+L_{f})Ve^{t}}{\Gamma(\alpha)} \left( 
\displaystyle\int_{0}^{t} e^{-qs}\Arrowvert y(s)\Arrowvert^{q}\, 
\mathrm{d}s\right)^{\dfrac{1}{q}}\\
&\qquad +\dfrac{(\lambda_{\max}(B)
+L_{f})Ve^{t}}{\Gamma(\alpha)} \left( \displaystyle\int_{0}^{t}  
e^{-qs}\Arrowvert y(s-\tau)\Arrowvert^{q}\, 
\mathrm{d}s\right)^{\dfrac{1}{q}}
\end{split}	
\end{equation*}
with $V$ given in \eqref{Psi}, which implies
\begin{equation}
\label{eq:pr:th1}
\begin{split}
\Arrowvert y(t)\Arrowvert
&\leq\Arrowvert\omega(0)\Arrowvert
+\dfrac{(\lambda_{\max}(A)+L_{f})Ve^{t}}{
\Gamma(\alpha)} \left( \displaystyle\int_{0}^{t} e^{-qs}
\Arrowvert y(s)\Arrowvert^{q}\, \mathrm{d}s\right)^{\dfrac{1}{q}}\\
&\qquad +\dfrac{(\lambda_{\max}(B)+L_{f})Ve^{t}}{\Gamma(\alpha)}
\left( \displaystyle\int_{-\tau}^{t}   
e^{-q(s+\tau)}\Arrowvert y(s)\Arrowvert^{q}\, \mathrm{d}s\right)^{\dfrac{1}{q}}.	
\end{split}	
\end{equation}
Now, by applying Jensen's inequality (Lemma~\ref{Jensen}) to \eqref{eq:pr:th1}, one gets
\begin{equation*}	
\begin{split}
\Arrowvert y(t)\Arrowvert^{q}
&\leq 3^{\frac{1}{\alpha}}\left[\Arrowvert\omega(0)\Arrowvert^{q}
+\dfrac{(\lambda_{\max}(A)+L_{f})^{q}V^{q}e^{qt}}{
\Gamma^{q}(\alpha)} \left( \displaystyle\int_{0}^{t} 
e^{-qs}\Arrowvert y(s)\Arrowvert^{q}\, \mathrm{d}s\right)\right.\\
&\qquad +\left.\dfrac{(\lambda_{\max}(B)+L_{f})^{q}V^{q}e^{qt}}{
\Gamma^{q}(\alpha)} \left( \displaystyle\int_{-\tau}^{t}       
e^{-q(s+\tau)}\Arrowvert y(s)\Arrowvert^{q}\, \mathrm{d}s\right)\right].	
\end{split}	
\end{equation*}
Hence, it follows that
$$
\Arrowvert y(t)\Arrowvert^{q}
\leq 3^{\frac{1}{\alpha}}\Arrowvert\omega(0)\Arrowvert^{q}
+\varPsi e^{qt} \displaystyle\int_{0}^{t} e^{-qs}\Arrowvert y(s)\Arrowvert^{q}\, \mathrm{d}s
+\dfrac{3^{\frac{1}{\alpha}}(\lambda_{\max}(B)+L_{f})^{q}V^{q}
e^{q(t-\tau)}}{\Gamma^{q}(\alpha)}
\displaystyle\int_{-\tau}^{0}       
e^{-qs}\Arrowvert y(s)\Arrowvert^{q}\, \mathrm{d}s,
$$
where $\varPsi$ is defined in \eqref{Psi}. This yields
$$
\Arrowvert y(t)\Arrowvert^{q}\leq3^{\frac{1}{\alpha}}\Arrowvert\omega\Arrowvert_{C}^{q}
+\varPsi e^{qt} \displaystyle\int_{0}^{t} 
e^{-qs}\Arrowvert y(s)\Arrowvert^{q}\, \mathrm{d}s
+e^{qt}\Arrowvert\omega\Arrowvert_{C}^{q}\varPhi 
$$
with $\varPhi$ given by \eqref{phi}. Hence,
$$
e^{-qt}\Arrowvert y(t)\Arrowvert^{q}\leq\left(
3^{\frac{1}{\alpha}}e^{-qs}+\varPhi\right)\Arrowvert\omega\Arrowvert_{C}^{q}
+\varPsi \displaystyle\int_{0}^{t} e^{-qs}\Arrowvert y(s)\Arrowvert^{q}\, \mathrm{d}s.
$$
Furthermore, by virtue of Gr\"onwall's inequality, one obtains that
$$
e^{-qt}\Arrowvert y(t)\Arrowvert^{q}\leq\left(
3^{\frac{1}{\alpha}}e^{-qs}+\varPhi\right)\Arrowvert\omega\Arrowvert_{C}^{q}
+\displaystyle\int_{0}^{t} \varPsi\left(3^{\frac{1}{\alpha}}e^{-qs}
+\varPhi\right)\Arrowvert\omega\Arrowvert_{C}^{q}e^{\varPsi(t-s)}\, \mathrm{d}s,
$$
which yields
\begin{equation}
\label{concl}
\Arrowvert y(t)\Arrowvert^{q}\leq
\dfrac{3^{\frac{1}{\alpha}}q+(3^{\frac{1}{\alpha}}\varPsi+q\varPhi+\varPsi\varPhi)
e^{(\varPsi+q) t}}{q+\varPsi}\Arrowvert\omega\Arrowvert_{C}^{q}.
\end{equation}
Consequently, from condition \eqref{cnd4} and inequality \eqref{concl}, 
we obtain the finite time stability of system \eqref{system2}.
\end{proof}

\begin{remark}
One notes that Theorem~\ref{theoa} gives only a sufficient condition that ensures 
the finite time stability of the time delay TFS \eqref{system2}. 
If this condition does not hold, we cannot conclude that \eqref{system2} is unstable.
\end{remark}

For the homogeneous case, we obtain from Theorem~\ref{theoa} the following result.

\begin{corollary}
Let $\xi, \varepsilon>0$ be given real numbers 
and consider $g=1+\alpha$ and $q=1+\dfrac{1}{\alpha}$. 
The homogeneous system \eqref{system1} is finite time stable 
with respect to $\{\xi, \varepsilon, J\}$, $\xi\leq\varepsilon$, if 
\begin{equation}
\label{cnd4c}
^{q}\sqrt{\dfrac{3^{\frac{1}{\alpha}}q
+(3^{\frac{1}{\alpha}}\varPsi+q\varPhi+\varPsi\varPhi)
e^{(\varPsi+q) t}}{q+\varPsi}}
\leq \dfrac{\varepsilon}{\xi}, \quad \forall t \in J,
\end{equation}
where
$$
\varPsi=\dfrac{3^{\frac{1}{\alpha}}\lambda_{\max}(A)^{q}
+\lambda_{\max}(B)^{q}e^{-q\tau})V^{q}}{\Gamma^{q}(\alpha)},
\quad 
V=\left(\dfrac{\Gamma(\alpha^{2})}{g^{\alpha^{2}}}\right)^{1/J}
$$
and
$$
\varPhi=\dfrac{3^{\frac{1}{\alpha}}\lambda_{\max}(B)^{q}(1
-e^{-\tau q})}{q\Gamma^{q}(\alpha)}V^{q}.
$$
\end{corollary}


\subsection{Time delay independent criterion}

Now, based on the Bellman--Gr\"onwall approach, we shall formulate a sufficient  
condition that enables the TFS~\eqref{system2} trajectories to stay within 
a priori given sets. 

\begin{theorem}
\label{theob}
Given real numbers $\xi>0$ and $\varepsilon>0$, 
the system~\eqref{system2} is finite time stable 
with respect to $\{\xi, \varepsilon, J\}$, $\xi\leq\varepsilon$, 
if $f$ satisfies condition \eqref{condLipch} and
\begin{equation}
\label{cnd1} 
\left(1
+\dfrac{(\lambda_{S}
+2L_{f})t^{\alpha}}{\Gamma(\alpha+1)}\right)
E_{\alpha}\left((\lambda_{S}+2L_{f})t^{\alpha}\right)
\leq \dfrac{\varepsilon}{\xi}, \quad \forall t \in J,
\end{equation}
where 
$\lambda_{S}= \lambda_{\max}(A)+\lambda_{\max}(B)$.
\end{theorem}

\begin{proof}
For all $t\in [0,T]$, the system \eqref{system2} admits a unique solution given by 	
$$
y(t)=\omega(0)e^{-\rho t}
+\dfrac{1}{\Gamma(\alpha)}\displaystyle \int_{0}^{t}
e^{-\rho (t-s)}e^{-\rho s}(t-s)^{\alpha-1}
\left[ A(s)y(s)+B(s)y(s-\tau)+ f(s,y(s),y(s-\tau))\right] \, \mathrm{d}s.
$$ 
Using the fact that $e^{-\rho s}\leq 1$ 
for all $s\in [0,t]$ and condition \eqref{condLipch} holds, then
\begin{equation}
\label{impa}
\Arrowvert y(t)\Arrowvert
\leq\Arrowvert\omega(0)\Arrowvert+\dfrac{1}{
\Gamma(\alpha)}\displaystyle \int_{0}^{t}e^{-\rho (t-s)}(t-s)^{\alpha-1}
\left[(\lambda_{\max}(A)+L_{f})\Arrowvert y(s)\Arrowvert
+\lambda_{\max}(B)+L_{f})\Arrowvert y(s-\tau)\Arrowvert\right] \, \mathrm{d}s.
\end{equation}
Setting $\Upsilon(t)=\underset{0\leq \zeta\leq t}{\sup}\Arrowvert y(\zeta)\Arrowvert$  
for all $t\in [0,T]$, one has  
\begin{equation}
\label{imp1}
\Arrowvert y(s-\tau)\Arrowvert \leq \Upsilon(s)+ \Arrowvert\omega\Arrowvert_{C}, 
\quad \forall s\in [0,t].
\end{equation}
Replacing \eqref{imp1} into inequality \eqref{impa}, it follows that
\begin{equation}
\label{impb}
\Arrowvert y(t)\Arrowvert
\leq\Arrowvert\omega(0)\Arrowvert+\dfrac{1}{\Gamma(\alpha)}
\displaystyle\int_{0}^{t}e^{-\rho(t-s)}(t-s)^{\alpha-1}(
\lambda_{S}+2L_{f})\left(\Upsilon(s)
+ \Arrowvert\omega\Arrowvert_{C}\right) \, \mathrm{d}s,
\end{equation}
which implies
\begin{equation}
\label{eq1}
\Arrowvert y(t)\Arrowvert\leq\Arrowvert\omega\Arrowvert_{C}
+\dfrac{(\lambda_{S}+2L_{f})t^{\alpha}}{\Gamma(\alpha+1)}
\Arrowvert\omega\Arrowvert_{C}
+\dfrac{(\lambda_{S}+2L_{f})}{\Gamma(\alpha)}
\displaystyle \int_{0}^{t}e^{-\rho(t-s)}(t-s)^{\alpha-1}
\Upsilon(s) \, \mathrm{d}s.
\end{equation}
Using the change of variable $x=t-s$, one obtains
\begin{equation}
\label{det1}
\Arrowvert y(t)\Arrowvert\leq\Arrowvert\omega\Arrowvert_{C}
+\dfrac{(\lambda_{S}+2L_{f})t^{\alpha}}{\Gamma(\alpha+1)}
\Arrowvert\omega\Arrowvert_{C}
+\dfrac{(\lambda_{S}+2L_{f})}{\Gamma(\alpha)}
\displaystyle \int_{0}^{t}e^{-\rho x} x^{\alpha-1}
\Upsilon(t-x) \, \mathrm{d}x.
\end{equation}
Also, by taking $t=\zeta$ in \eqref{det1} with $\zeta \in [0,t]$ 
and using $\zeta^{\alpha}\leq t^{\alpha}$, we get
\begin{equation}
\label{eq2}
\Arrowvert y(\zeta)\Arrowvert
\leq\left[1+\dfrac{(\lambda_{S}+2L_{f})t^{\alpha}}{
\Gamma(\alpha+1)}\right]\Arrowvert\omega\Arrowvert_{C}
+\dfrac{(\lambda_{S}+2L_{f})}{\Gamma(\alpha)}
\displaystyle \int_{0}^{\zeta}e^{-\rho x} x^{\alpha-1}
\Upsilon(\zeta-x) \, \mathrm{d}x.
\end{equation}
Since the function $\Upsilon$ is nonnegative, 
it follows that $\displaystyle \int_{0}^{t}e^{-\rho x} x^{\alpha-1}
\Upsilon(t-x) \, \mathrm{d}x$ is an increasing function 
with respect to $t\geq0$, which implies that 
$\displaystyle \int_{0}^{\zeta}e^{-\rho x} x^{\alpha-1}
\Upsilon(\zeta-x)  \, \mathrm{d}x\leq\displaystyle 
\int_{0}^{t}e^{-\rho x} x^{\alpha-1}
\Upsilon(t-x) \, \mathrm{d}x$ and
\begin{equation*}
\Arrowvert y(\zeta)\Arrowvert\leq\left[1+\dfrac{(\lambda_{S}+2L_{f})
t^{\alpha}}{\Gamma(\alpha+1)}\right]
\Arrowvert\omega\Arrowvert_{C}+\dfrac{(\lambda_{S}+2L_{f})}{\Gamma(\alpha)}
\displaystyle \int_{0}^{t}e^{-\rho x} x^{\alpha-1}
\Upsilon(t-x)\, \mathrm{d}x.
\end{equation*}
It follows that
\begin{equation*}
\Upsilon(t)=\underset{0\leq \zeta\leq t}{\sup}\Arrowvert y(\zeta)
\Arrowvert\leq\left[1+\dfrac{(\lambda_{S}+2L_{f})
t^{\alpha}}{\Gamma(\alpha+1)}\right]\Arrowvert\omega\Arrowvert_{C}
+\dfrac{(\lambda_{S}+2L_{f})}{\Gamma(\alpha)}
\displaystyle \int_{0}^{t}e^{-\rho(t-s)}(t-s)^{\alpha-1}
\Upsilon(s)\, \mathrm{d}s.
\end{equation*}
Now, let 
$f(t)=\left[1+\dfrac{(\lambda_{S}+2L_{f})t^{\alpha}}{
\Gamma(\alpha+1)}\right]\Arrowvert\omega\Arrowvert_{C}$,
which is a nondecreasing function. By applying Lemma~\ref{Gronwall} 
with $h(t)=\dfrac{(\lambda_{S}+2L_{f})}{\Gamma(\alpha)}$, 
we get
$$
\Arrowvert y(t)\Arrowvert\leq\Upsilon(t)
\leq \Arrowvert\omega\Arrowvert_{C} \left[1
+\dfrac{(\lambda_{S}+2L_{f})t^{\alpha}}{\Gamma(\alpha+1)}
\right]E_{\alpha}\left(
(\lambda_{S}+2L_{f}) t^{\alpha}\right).
$$
Then, by virtue of \eqref{Cnd1} and \eqref{cnd1}, one deduces that 
$$
\Arrowvert y(t) \Arrowvert < \varepsilon, 
\quad \forall t \in J=[0,T],
$$
which proves the finite time stability of 
the nonhomogeneous  TFS \eqref{system2}.
\end{proof}

\begin{remark}
The condition \eqref{cnd1} can be written, in equivalent way, as follows:	
\begin{equation*}
\left(1+\dfrac{(\Arrowvert A \Arrowvert + \Arrowvert B
\Arrowvert+2L_{f})t^{\alpha}}{\Gamma(\alpha+1)}\right)
E_{\alpha}\left((\Arrowvert A \Arrowvert 
+ \Arrowvert B \Arrowvert+2L_{f}) t^{\alpha}\right)
\leq \dfrac{\varepsilon}{\xi}, \quad \forall t \in J.
\end{equation*}	
\end{remark}

In the homogeneous case, we obtain from Theorem~\ref{theob} the following result.

\begin{corollary}
\label{Cor1}
Given real numbers $\xi>0$ and $\varepsilon>0$, the homogeneous 
system \eqref{system1} is finite time stable with respect to 
$\{\xi, \varepsilon, J\}$, $\xi\leq\varepsilon$, if 
\begin{equation}
\label{cnd1c} 
\left(1+\dfrac{\lambda_{S}t^{\alpha}}{\Gamma(\alpha+1)}\right)
E_{\alpha}\left(\lambda_{S} t^{\alpha}\right)
\leq \dfrac{\varepsilon}{\xi}, \quad \forall t \in J.
\end{equation}	
\end{corollary}

\begin{remark} 
From Corollary~\ref{Cor1}, if we let $\rho=0$ in system~\eqref{system1}, 
then one retrieves the condition 
\begin{equation*}
\left(1+\dfrac{\lambda_{S}t^{\alpha}}{\Gamma(\alpha+1)}\right)
E_{\alpha}\left(\lambda_{S}t^{\alpha}\right)
\leq \dfrac{\varepsilon}{\xi}, 
\quad \forall t \in J,
\end{equation*}
for the finite time stability of the Caputo 
fractional order time-delay system
\begin{equation*}
\left\{
\begin{array}{ll}
{}^C \!D_{0}^{\alpha}y(t)=Ay(t)+By(t-\tau), & t\in[0,T],\\
y(t)=\omega(t),  & t\in[-\tau,0],
\end{array}
\right.
\end{equation*}
where ${}^C \!D_{0}^{\alpha}$ is the Caputo fractional derivative 
of order $\alpha$. This result is proved in \cite{Lazarevi}.
\end{remark}


\section{Illustrative Examples}
\label{sec:4}

In this section, we present two expository examples 
in order to illustrate our previous results.

Our first example gives a situation where Theorem~\ref{theoa} 
allows us to conclude that the given TFS
is finite time stable while Theorem~\ref{theob} fails.

\begin{example}
\label{ex1} 
Consider the nonhomogeneous tempered fractional system with time delay
\begin{equation}
\label{system2a}
\left\{
\begin{array}{ll}
{}^T\! D_{0}^{0.3,0.8}y(t)=e^{-0.8}t(Ay(t)
+By(t-\tau)+2\sin y(t)-3\sin y(t-\tau)), & t\in[0,3],\\
y(t)=[0\quad 0]^{T},  & t\in[-0.2,0],
\end{array}
\right.
\end{equation}
where
$$
A=\begin{pmatrix}
-2&0\\
0&-2\\
\end{pmatrix},~B=\begin{pmatrix}
3&-4\\
0&0\\
\end{pmatrix}~ \text{ and }~ y=\begin{pmatrix}
y_{1}\\
y_{2}\\	
\end{pmatrix}\in \mathbb{R}^{2}.
$$
According to system \eqref{system2}, one has  $\alpha=0.3$, 
$\rho=0.8$, $\tau=0.2$, $T=3$, $\omega(t)=[0\quad 0]^{T}$ and 
$$
f(t,y(t),y(t-\tau))=2\sin y(t)-3\sin y(t-\tau).
$$
For all  $y,z : [-0.2, 3]\longrightarrow\mathbb{R}^{2}$, one has 
\begin{equation*}	
\begin{split}
\Arrowvert f(t,y(t),y(t-\tau))- f(t,z(t),z(t-\tau))  \Arrowvert 
&=\Arrowvert 2(\sin y(t)-\sin z(t))-3(\sin y(t-\tau)-\sin z(t-\tau))\Arrowvert \\
&\leq 2\left[\Arrowvert\sin y(t)-\sin z(t)\Arrowvert 
+ 3\Arrowvert \sin y(t-\tau)-\sin z(t-\tau)\Arrowvert\right]\\
&\leq
3\left[\Arrowvert y(t)- z(t)\Arrowvert + \Arrowvert y(t-\tau)- z(t-\tau)\Arrowvert
\right],\quad t\in[0,3].
\end{split}	
\end{equation*}
It follows that the nonlinear function $f$ satisfies 
the  condition \eqref{condLipch} with $L_{f}=3$ and $f(t,0,0)=0$.
Also, one needs to check the finite time stability of system \eqref{system2a} 
with regard to $J=[0,3]$, $\xi=0.01$, and $\varepsilon=0.6$.
From system \eqref{system2a}, one gets	
$$\Arrowvert \omega \Arrowvert_{C}	<0.01,$$	
$$\lambda_{\max}(A)= 2, \lambda_{\max}(B)=5~\text{ and }~\lambda_{S}=7.$$	
Hence, numerically, we have $q=4.3333$, $\varPsi=0.4945$,
$\varPhi=0.1201$ and 
\begin{equation*}
C_{1}(t)=^{q}\sqrt{8.0658+4.1088e^{(4.8279) t}}\leq 60
=\dfrac{\varepsilon}{\xi}, \quad \forall t\in [0,3]. 
\end{equation*}
Therefore, condition \eqref{cnd4} holds, 
as it is illustrated in Figure~\ref{fig1}. Then, from Theorem~\ref{theoa}, 
we deduce that system \eqref{system2a} is finite time stable with respect to 
$\{\xi=0.01, \varepsilon=0.6, J=[0,3]\}$. However, we cannot arrive 
at the same conclusion using Theorem~\ref{theob}, since
$$ 
C_{2}(t)= \left(1+14.4847\, t^{0.3}\right)
E_{0.3}\left(13 \, t^{0.3}\right)\leq 60,
$$
is not satisfied for all $t \in [0,3]$, as shown in Figure~\ref{fig2}, 
which means that condition \eqref{cnd1} does not hold. 
\begin{figure}[H]
\centering
\includegraphics[scale=0.6]{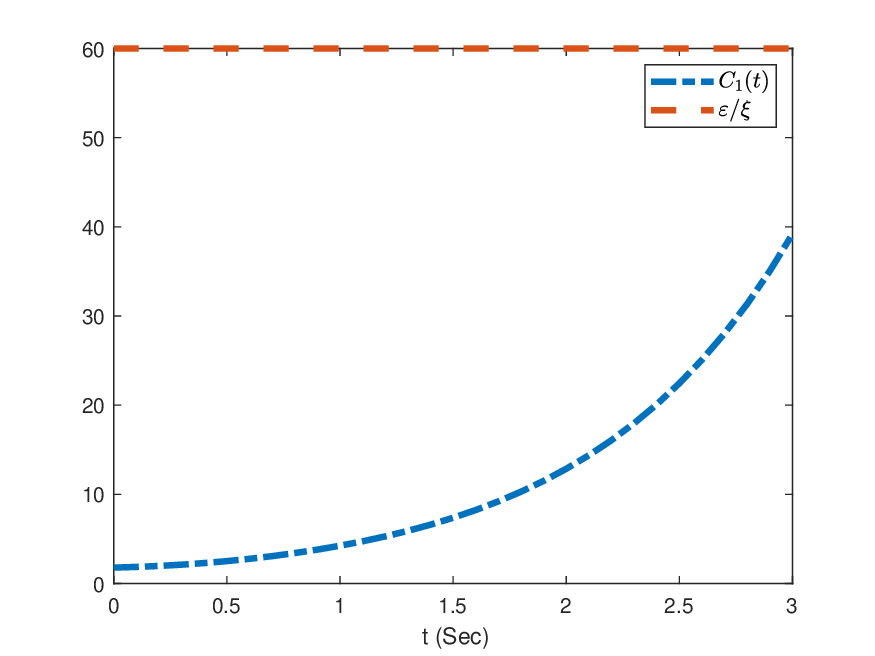}
\caption{Condition \eqref{cnd4}, over $[0, 3]$, 
for $\alpha=0.3$, $\rho= 0.8$, 
$\xi=0.01$ and $\varepsilon=0.6$.}
\label{fig1}
\end{figure} 
\begin{figure}[H]
\centering
\includegraphics[scale=0.6]{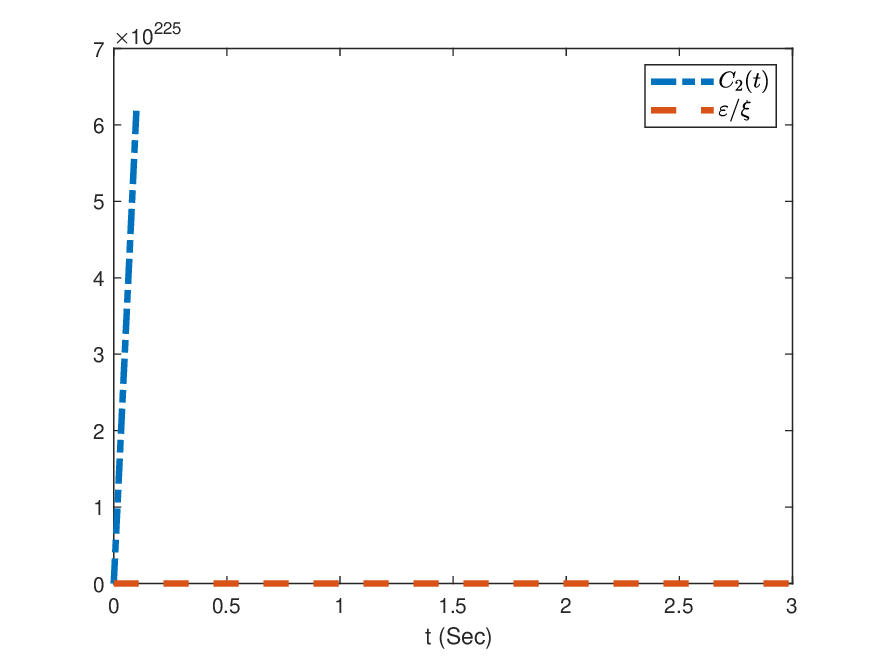}
\caption{Condition \eqref{cnd1}, over $[0, 3]$, for $\alpha=0.3$, 
$\rho= 0.8$, $\xi=0.01$ and $\varepsilon=0.6$.}
\label{fig2}
\end{figure} 
\end{example}


In contrast with Example~\ref{ex1}, now we consider a problem where 
Theorem~\ref{theob} allows us to deduce the finite time stability 
of the system while Theorem~\ref{theoa} does not.

\begin{example} 
Let us consider the following nonhomogeneous TFS time delay system:
\begin{equation}
\label{system2exp}
\left\{
\begin{array}{ll}
{}^T\! D_{0}^{0.5,0.5}y(t)
=e^{-0.5 t}(Ay(t)+By(t-\tau)+0.03(y(t)+\sin(y(t-\tau)), 
& t\in[0,4],\\
y(t)=[0.01\cos(\pi t)\quad 0.01]^{T},  
& t\in[-0.2,0],
\end{array}
\right.
\end{equation}
with 
$$
A=\begin{pmatrix}
0&0.2\\
-0.15&0\\
\end{pmatrix},
\quad
B=
\begin{pmatrix}
-0.1&0\\
0&-0.09\\
\end{pmatrix}
\text{ and } 
y=
\begin{pmatrix}
y_{1}\\
y_{2}\\	
\end{pmatrix}
\in \mathbb{R}^{2}.
$$
According to system \eqref{system2}, one has $\alpha=\rho=0.5$, $\tau=0.2$, $T=4$,
$\omega(t)=[0.01\cos(\pi t)\quad 0.01]^{T}$ and 
$$
f(t,y(t),y(t-\tau))=0.03\left(y(t)+\sin(y(t-\tau))\right).
$$
One has 
\begin{equation*}	
\begin{split}
\Arrowvert f(t,y(t),y(t-\tau))- f(t,z(t),z(t-\tau))  \Arrowvert 
&\leq 0.03\Arrowvert y(t)- z(t)\Arrowvert+\Arrowvert
\sin y(t-\tau)-\sin z(t-\tau)\Arrowvert \\
&\leq 0.03\left[\Arrowvert y(t)- z(t)\Arrowvert 
+ \Arrowvert y(t-\tau)- z(t-\tau)\Arrowvert \right], \quad t\in[0,4],
\end{split}	
\end{equation*}
for all $y,z : [-0.2, 4]\longrightarrow\mathbb{R}^{2}$. It follows that condition 
\eqref{condLipch} holds with $L_{f}=0.03$ and $f(t,0,0)=0$.
	
We need to check the finite time stability of system \eqref{system2exp} 
with regard to  $J=[0,4]$, $\xi=0.02$, $\varepsilon=0.2$.
From system \eqref{system2exp}, it follows that
$$
\Arrowvert \omega \Arrowvert_{C} < 0.02
$$
and 
$$
\lambda_{\max}(A)= 0.2, \lambda_{\max}(B)=0.1,\lambda_{S}=0.3.
$$
On the other hand, for $m\in \mathbb{N}$, one has
$$
\Gamma\left(m+\dfrac{1}{2}\right)
=\dfrac{\sqrt{(\pi)\Gamma(2m+1)}}{2^{2m}\Gamma(m+1)}
=\dfrac{\sqrt{(\pi)(2n)!}}{2^{2n}m!}.
$$
Therefore, for $m=1$, we have
$$
\Gamma\left(1+\dfrac{1}{2}\right)
=\dfrac{\sqrt{(\pi)2!}}{2^{2}1!}=0.886.
$$
Moreover, by using the numeric computing environment
\textsf{MATLAB}, we obtain that 
\begin{equation*}
C_{2}(t)=\left(1+0.4063\,t^{0.5}\right)
E_{0.5}\left(0.36\,t^{0.5}\right)
\leq 10=\dfrac{\varepsilon}{\xi}, 
\quad \forall t \in [0, 4],
\end{equation*}
which means that condition \eqref{cnd1} holds, 
as it is shown in Figure~\ref{fig3}.
Then, from Theorem~\ref{theob}, we deduce that system 
\eqref{system2exp} is finite time stable with respect to 
$\{\xi=0.02, \varepsilon=0.2, J=[0,4]\}$.
	
However, by using Theorem~\ref{theoa}, 
we cannot conclude 
that the system \eqref{system2exp}
is finite time stable. Indeed, for 
$[3.5, 4]\subset [0, 4]$ one has $\varPsi=0.0075$, 
$\varPhi=1.8486\times 10^{-4}$ and
\begin{equation*}
C_{1}(t)=^{3}\sqrt{
8.9776+0.0226e^{(3.0075) t}}> 10, 
\end{equation*}
which means that condition \eqref{cnd4} does not hold as it is 
illustrated in Figure~\ref{fig4}.
\begin{figure}[H]
\centering
\includegraphics[scale=0.6]{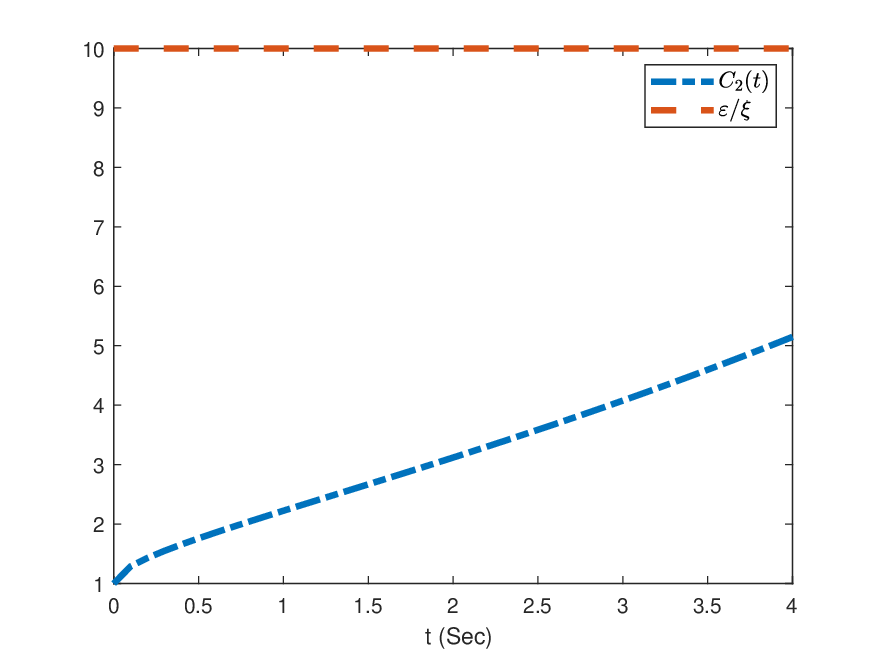}
\caption{Condition \eqref{cnd1}, over $[0, 4]$, 
for $\alpha=\rho= 0.5$, $\xi=0.02$ and $\varepsilon=0.2$.}
\label{fig3}
\end{figure} 
\begin{figure}[H]
\centering
\includegraphics[scale=0.6]{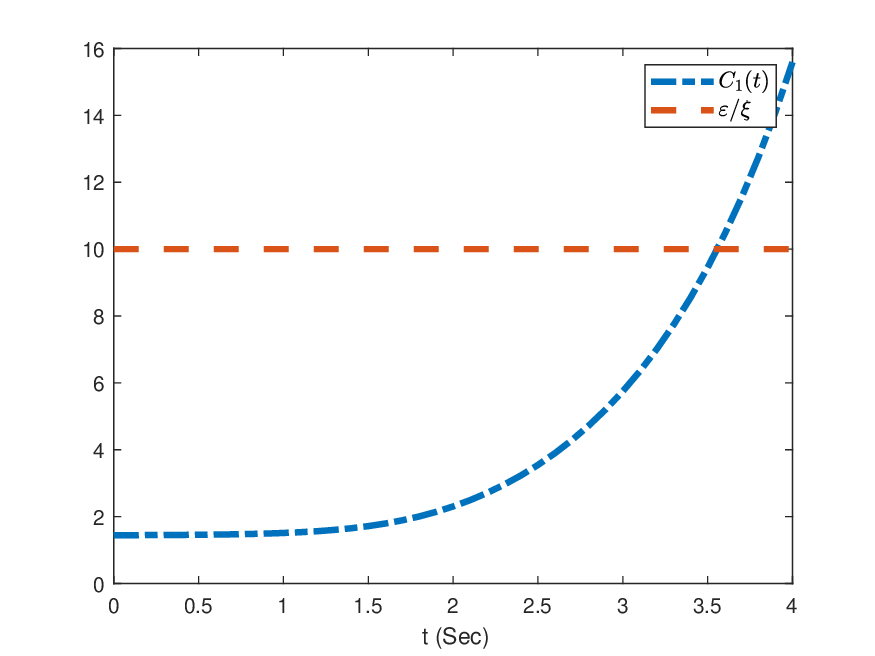}
\caption{Condition \eqref{cnd4}, over $[0, 4]$, 
for $\alpha=\rho= 0.5$, $\xi=0.02$ and $\varepsilon=0.2$.}
\label{fig4}
\end{figure}
\end{example}

Our examples show that Theorems~\ref{theoa} and \ref{theob}  
are different and both useful, 
depending on the systems under study.


\section{Conclusion}
\label{sec:5}

For some engineering systems whose operation is time limited 
and should be done within prescribed bounds on system variables, 
the only meaningful stability concept is finite time stability. 
Since the concept of a change given in terms of the tempered
fractional derivative (TFD) is more appropriate for some specific applications,
in this work we provided two finite time stability test procedures 
for fractional differential equations with time delays 
involving the TFD. One stability criterion
depends on the time delay while the second one 
is delay independent. We used mainly two different approaches
for nonhomogeneous time delay TFSs over a finite time interval:
(i) one is based on H\"older's and Jensen's inequalities; 
(ii) the second one on Bellman--Gr\"onwall method using 
the tempered fractional Gr\"onwall inequality. 
The effectiveness of the proposed procedures was illustrated 
through two numerical examples,
showing that the obtained criteria
are different and relevant. Our developed stability results 
may be applied to investigate the stability, over a finite time, 
of different mathematical delayed models, e.g., 
neural networks with a bounded activation function 
or tuberculosis epidemic models.

The generalized kernel idea started 
with Boltzmann in 1874 \cite{Boltzmann}
and can be seen, e.g., in \cite{MR2404728}.
As future work, we plan to investigate 
necessary conditions for the finite time stability 
of TFSs \eqref{system2} and also to analyze 
their finite time stabilization,
developing numerical methods to approximate the solution 
of the considered problems and to study the stability
of fractional delayed systems with more general types of kernels.


\section*{Acknowledgments}

The authors would like to thank three anonymous reviewers 
and an Associate Editor for their recommendations and suggestions, 
that helped them to improve the initial submitted manuscript.


\section*{Funding}

Zitane and Torres are supported by The Center for Research and
Development in Mathematics and Applications (CIDMA)
through the Portuguese Foundation for Science and Technology 
(FCT -- Funda\c{c}\~{a}o para a Ci\^{e}ncia e a Tecnologia),
projects UIDB/04106/2020 and UIDP/04106/2020.


\section*{Disclosure} 

The authors report there are no competing interests to declare. 


\section*{Data Availability} 

Not applicable.



\end{document}